\documentclass{sig-alternate}
\usepackage{bm,bbm}

\def\N{\mathbbm N}
\def\Z{\mathbbm Z}
\def\Q{\mathbbm Q}

\def\Pf{\operatorname{Pf}}
\def\defterm#1{{\emph{#1}}}
\def\sgn{\operatorname{sgn}}

\renewcommand{\le}{\leqslant}\renewcommand{\leq}{\leqslant}
\renewcommand{\geq}{\geqslant}

\newtheorem{theorem}{Theorem}
\newtheorem{proposition}[theorem]{Proposition}
\newtheorem{lemma}[theorem]{Lemma}
\newtheorem{conjecture}[theorem]{Conjecture}
\newtheorem{corollary}[theorem]{Corollary}

\begin{document}
%
\conferenceinfo{ISSAC'12, July 22--25, 2012,}{Grenoble, France}

\title{Zeilberger's Holonomic Ansatz for Pfaffians}
\numberofauthors{2}
\author{
\alignauthor
Masao Ishikawa\titlenote{
  Partially supported by the Ministry of Education, Science, Sports and Culture of Japan, 
  Grant-in-Aid for Scientific Research (C), 21540015.}\\
  \affaddr{Department of Mathematics}\\
  \affaddr{University of the Ryukyus}\\
  \affaddr{Nishihara, Okinawa 901-0213, Japan}\\
  \email{ishikawa@edu.u-ryukyu.ac.jp}
\alignauthor
Christoph Koutschan\titlenote{
  Supported by the Austrian Science Fund (FWF): P20162-N18.}\\
  \affaddr{MSR-INRIA Joint Centre}\\
  \affaddr{INRIA-Saclay}\\
  \affaddr{91893 Orsay Cedex, France}\\
  \email{koutschan@risc.jku.at}
}
\date{\today}

\maketitle

\begin{abstract}
A variation of Zeilberger's holonomic ansatz for symbolic determinant
evaluations is proposed which is tailored to deal with Pfaffians.  The
method is also applicable to determinants of skew-symmetric matrices,
for which the original approach does not work. As Zeilberger's
approach is based on the Laplace expansion (cofactor expansion) of the
determinant, we derive our approach from the cofactor expansion of the
Pfaffian.  To demonstrate the power of our method, we prove, using
computer algebra algorithms, some conjectures proposed in the paper
``Pfaffian decomposition and a Pfaffian analogue of $q$-Catalan Hankel
determinants'' by Ishikawa, Tagawa, and Zeng. A minor summation
formula related to partitions and Motzkin paths follows as a
corollary.
\end{abstract}

\category{G.2.1}{Discrete Mathematics}{Combinatorics}[Recurrences and difference equations]
\category{G.4}{Mathematical Software}{Algorithm design and analysis}

\terms{Algorithms, Theory}

\keywords{Pfaffian, determinant, minor, holonomic systems approach, 
WZ theory, symbolic summation, computer proof, Motzkin number}

%
%
\section{Introduction}
Pfaffians are a very important concept in combinatorics and in
physics, for example, for the enumeration of plane partitions,
Kasteleyn's method for the dimer models, etc.  We introduce an
algorithmic method for evaluating Pfaffians which allows us to solve
such problems automatically by computer; we demonstrate its
applicability by proving a few conjectures
in~\cite{IshikawaTagawaZeng10}, concerning Pfaffians of interesting
combinatorial numbers.  Our approach is a variation of Zeilberger's
holonomic ansatz for evaluating determinants, which we recall
in the following, for sake of self-containedness.

In~\cite{Zeilberger07}, Zeilberger proposed an algorithmic approach
for evaluating and/or producing rigorous proofs of determinant
evaluations of the form
\[
  \det\,(\check a_{i,j})_{1\leq i,j\leq n} =\check  b_n
\]
(we use checked letters here to avoid confusion with the quantities
introduced in Section~\ref{sec.pfaffians}).  The goal is achieved in a
completely automatic fashion, using computer algebra algorithms for
guessing recurrences and symbolic summation.
The key point is to \emph{guess}~\cite{Kauers09} a suitable (implicit)
description of an auxiliary function~$\check c_{n,j}$ and then prove
that it satisfies the three identities
\begin{alignat}{2}
\check c_{n,n} & = 1 && (n\geq 1),
  \tag{$\check 1$}\label{eq:cofactor1}\\
\sum^{n}_{j=1}\check c_{n,j} \check a_{i,j} & = 0 && (1\leq i < n),
  \tag{$\check 2$}\label{eq:cofactor2}\\
\sum^{n}_{j=1}\check c_{n,j} \check a_{n,j} & = \frac{\check b_n}{\check b_{n-1}} &\qquad& (n\geq 1).
  \tag{$\check 3$}\label{eq:cofactor3}
\end{alignat}
The determinant evaluation follows as a consequence, using Laplace
expansion w.r.t. the last row and induction on~$n$.

In principle, the approach is applicable if the matrix is never
singular, i.e., if $\check b_n\neq0$ for all $n\geq0$. But in order to
turn Identities~\eqref{eq:cofactor1}--\eqref{eq:cofactor3} into
routinely provable tasks, Zeilberger additionally requires that the
matrix entries~$\check a_{i,j}$ constitute a bivariate holonomic
sequence and that the ratios of two consecutive determinants $\check
b_n/\check b_{n-1}$ form a univariate holonomic (P-finite) sequence
(in other words, $\check b_n$ is required to be what one could call
\emph{hyper-holonomic}). This is the reason why he termed his approach
the \emph{holonomic ansatz}~\cite{Zeilberger07,Zeilberger90}.  But
still, even if all these conditions are satisfied, Zeilberger's
holonomic ansatz is not guaranteed to succeed, because it relies on
the fact that the auxiliary function~$\check c_{n,j}$ (that is, the
cofactors of the Laplace expansion with respect to the last row of the
$n\times n$ matrix, divided by the determinant of the
$(n-1)\times(n-1)$ matrix) turns out to be holonomic, too. This may be
the case or not. If one is lucky, i.e., if $\check c_{n,j}$ satisfies
sufficiently many linear recurrence equations with polynomial
coefficients and therefore is holonomic, then the holonomic machinery
will produce a P-finite recurrence for the sum on the left-hand side
of~\eqref{eq:cofactor3}.  Such a recurrence can then be used to prove
a (conjectured) determinant evaluation~$\check b_n$ by substituting
the ratio~$\check b_n/\check b_{n-1}$ into this recurrence and
comparing initial values, or even, if the recurrence is not too
complicated, to solve it explicitly and obtain a closed form for the
determinant.

For a more detailed description and justification of the holonomic
ansatz, see~\cite{Zeilberger07,KoutschanKauersZeilberger10}. We also
recommend the beautiful essay~\cite{Krattenthaler99} for the reader
who is interested in determinant and Pfaffian evaluations in general.

In the present paper, we introduce a variation of Zeilberger's method
that is tailored particularly for Pfaffians. Recall that Pfaffians
are defined only for skew-symmetric matrices and that the square of
the Pfaffian equals the determinant. As a trivial consequence, our
approach addresses determinants of skew-symmetric matrices as well.
Clearly Zeilberger's holonomic ansatz cannot be applied to
skew-symmetric matrices, since the determinant in this case vanishes
whenever the dimension is odd.  Another extension of Zeilberger's
ansatz, the so-called \emph{double-step method}, is applicable to
matrices that are zero either for even or odd
dimensions~\cite{KoutschanThanatipanonda12}.  Concerning the
evaluation of determinants only (not Pfaffians), the double-step
method is more general, as it does not assume skew-symmetry, but at
the same time much more complicated and less efficient than our
approach for Pfaffians.

In the next section, we state the cofactor expansion of the Pfaffian
and use it to develop our algorithmic approach for dealing with
evaluations of Pfaffians; this means proof and/or discovery, as in
Zeilberger's approach for determinants.  In the following
Sections~\ref{sec.Motz}--\ref{sec.Nara} this method is used to solve
some open problems posed in~\cite{IshikawaTagawaZeng10}. The details
of our computer proofs are provided as supplementary electronic
material on the webpage
\begin{center}
\texttt{http://www.risc.jku.at/people/ckoutsch/pfaffians/}
\end{center}
in form of a Mathematica notebook. It is supposed to enable the reader
to reproduce our results and do further experiments.  In
Section~\ref{sec.Appl} we use our results (Theorem~\ref{thm.pfMotz})
to prove an interesting minor summation formula where the sum ranges
over certain partitions and the matrix entries are variations of
Motzkin numbers. We conclude this article by posing some open problems
as future challenges.

%
\section{Pfaffians}\label{sec.pfaffians}
Let $n$ be a positive integer and let 
 $A=(a_{i,j})_{1\le i,j\le 2n}$ be 
a $2n$ by $2n$ skew-symmetric matrix, i.e., $a_{j,i}=-a_{i,j}$, 
whose entries $a_{i,j}$ are in a commutative ring.
Note that it is completely determined by its upper 
triangular entries~$a_{i,j}$ for $1\leq i<j\leq 2n$.
The \defterm{Pfaffian} $\Pf(A)$ of $A$ is defined by
\begin{equation*}
\Pf(A)=\sum \epsilon(\sigma_{1},\sigma_{2},\hdots,\sigma_{2n-1},\sigma_{2n})\,
a_{\sigma_{1}\sigma_{2}} \dots a_{\sigma_{2n-1}\sigma_{2n}}.
\end{equation*}
where the summation is over all partitions 
\[
  \{\{\sigma_{1},\sigma_{2}\},\hdots,\{\sigma_{2n-1},\sigma_{2n}\}\}
\]
of $[2n]=\{1,2,\dots,2n\}$ into two-elements subsets,
and where $\epsilon(\sigma_{1},\sigma_{2},\hdots,\sigma_{2n-1},\sigma_{2n})$ 
denotes the sign of the permutation
\begin{equation*}
\begin{pmatrix}
1&2&\cdots&2n-1&2n\\
\sigma_{1}&\sigma_{2}&\cdots&\sigma_{2n-1}&\sigma_{2n}
\end{pmatrix}.
\end{equation*}
A permutation $(\sigma_{1},\sigma_{2},\hdots,\sigma_{2n-1},\sigma_{2n})$ 
which arises from a partition of $[2n]$ into $2$-elements blocks is called a \defterm{perfect matching} or a \defterm{$1$-factor}.
For any permutation $\pi$ of $[2n]$, let $A^{\pi}=(a_{\pi(i)\pi(j)})$
denote the skew-symmetric matrix obtained by the natural action of $\pi$ on both rows and columns.
From the  definition above it is easy to see that 
\begin{equation*}
\Pf(A^{\pi})=\sgn\pi\,\Pf(A).
\end{equation*}
Hence, if any two rows and/or columns are coinciding in~$A$,
the Pfaffian $\Pf(A)$ of~$A$ vanishes. It is a well-known fact
that $\Pf(A)^2=\det(A)$.
Now let $I=\{i_1,\dots,i_{r}\}$ be an $r$-element subset of $[2n]$:
we denote by 
\[
  A(I)=A(i_1,\dots,i_{r})
\]
the skew-symmetric $(2n-r)\times (2n-r)$ matrix obtained from~$A$ by removing
the rows $i_1,\dots,i_r$ and the columns $i_1,\dots,i_r$.
Also let us define $\Gamma_{i,j}$ for $1\leq i,j\leq 2n$ by
\[
\Gamma_{i,j}=\begin{cases}
(-1)^{j-i-1}\Pf A(i,j)
&\text{ if $i<j$,}\\
(-1)^{i-j}\Pf A(j,i)
&\text{ if $j<i$,}\\
0
&\text{ if $i=j$.}
\end{cases}
\]
The Laplace expansion formula for Pfaffians reads as follows.
\begin{proposition}
\label{thm.Laplace}
Let $A=(a_{i,j})_{1\leq i,j\leq 2n}$ be a skew-sym\-metric matrix, 
and $\Gamma_{i,j}$ be as above.
Then we have
\[
  \sum_{k=1}^{2n}a_{i,k}\Gamma_{j,k} =
  \sum_{k=1}^{2n}a_{k,i}\Gamma_{k,j} = \delta_{i,j}\Pf A.
\]
\end{proposition}
\begin{proof}
This statement and its proof are found in~\cite{IshikawaWakayama10}.
\end{proof}

Hence if one puts $b_{2n}=\Pf A=\Pf(a_{i,j})_{1\leq i,j\leq2n}$ and 
$c_{2n,j}={\Gamma_{j,2n}}/{\Gamma_{2n-1,2n}}$ for
$1\leq i\leq 2n-1$,
then Proposition~\ref{thm.Laplace} implies that
$c_{2n,j}$ satisfies the following three identities:
\begin{alignat}{2}
c_{2n,2n-1} & = 1 && (n\geq 1), 
\label{eq.1}\\
\sum_{i=1}^{2n-1} c_{2n,i} a_{i,j} & = 0 && (1\leq j<2n), 
\label{eq.2}\\
\sum_{i=1}^{2n-1} c_{2n,i} a_{i,2n} & = \frac{b_{2n}}{b_{2n-2}} &\qquad& (n\geq 1).
\label{eq.3}
\end{alignat}
Conversely, one easily sees that the bivariate sequence~$c_{2n,j}$ is
uniquely characterized by Equations~\eqref{eq.1} and~\eqref{eq.2}, and
we can regard \eqref{eq.1}--\eqref{eq.3} as the formulation analogous
to \eqref{eq:cofactor1}--\eqref{eq:cofactor3} in order to evaluate the
Pfaffian $\Pf A=\Pf(a_{i,j})_{1\leq i,j\leq2n}$.  For this purpose,
one first has to guess a suitable implicit (i.e., holonomic)
description of the function~$c_{2n,i}$ and then show that it indeed
satisfies the above identities. Induction on~$n$ concludes the proof.
The methodology is illustrated in detail by an example in
Section~\ref{sec.Motz}.

Identities~\eqref{eq.1}, \eqref{eq.2}, and~\eqref{eq.3} can be proven
algorithmically in the spirit of the \emph{holonomic systems
  approach}~\cite{Zeilberger90}.  In the following sections the
software package~\texttt{HolonomicFunctions}~\cite{Koutschan10b} which
runs under the computer algebra system Mathematica is employed for
carrying out the necessary computations.  The
thesis~\cite{Koutschan09} describes the theoretical background and the
algorithms implemented therein.

%
\section{A Motzkin Number Pfaffian}\label{sec.Motz}
This section gives a detailed computer proof of a Pfaffian involving
the Motzkin numbers. It is stated as an open problem
in~\cite{IshikawaTagawaZeng10}, see Formula (6.3) there. The Motzkin
numbers~$M_n$ can be obtained by the formula
\[
  M_n=\sum_{k=0}^n \frac{1}{k+1}\binom{n}{2k}\binom{2k}{k}
     ={}_2F_1\!\left(\genfrac{}{}{0pt}{}{-\frac{n}{2},\frac{1-n}{2}}{2};4\right),
\]
where ${}_2F_1$ stands for the Gau\ss\ hypergeometric function.
They count Motzkin paths from $(0,0)$ to $(n,0)$; recall that a
Motzkin path is a path in the lattice~$\N_0^2$ that uses only the steps
\[
  U = (1,1),\qquad H = (1,0),\qquad D = (1,-1)
\]
and never runs below the horizontal axis
(see~\cite{DonagheyShapiro77}).

\begin{theorem}\label{thm.pfMotz}
For all integers $n\geq1$ the following identity holds:
\begin{equation}\label{eq.pfMotz}
  \Pf\big((j-i)M_{i+j-3}\big)_{1\leq i,j\leq 2n} = \prod_{k=0}^{n-1}(4k+1).
\end{equation}
\end{theorem}
\begin{proof}
The proof is split into several parts which are presented in the
Sections~\ref{sec.implicit}--\ref{sec.eq3} below. The details of the
computations are contained in the supplementary electronic material
mentioned in the introduction.
\end{proof}

\subsection{Implicit Description for {\large $c_{2n,i}$}}
\label{sec.implicit}
The first step is to determine the auxiliary function~$c_{2n,i}$ that
appears in identities~\eqref{eq.1}--\eqref{eq.3}, where now
\[
  a_{i,j}=(j-i)M_{i+j-3}.
\]
Using the method of guessing, as implemented in the Mathe\-matica
package \texttt{Guess}~\cite{Kauers09}, one comes up with an implicit
description of this unknown function, namely the following three
linear recurrence equations with polynomial coefficients:
{\setlength{\arraycolsep}{0pt}
\begin{eqnarray}
&&\begin{array}{l}
  (i-1)(2n-3)(4n-7)c_{2n,i}=\\
  \quad -(2n+i-4)(8in-8i-8n^2+6n+3)c_{2(n-1),i-1}+{}\\
  \quad (i-1)(16in-16i+8n^2-34n+27)c_{2(n-1),i}+{}\\
  \quad 24i(i-1)(n-1)c_{2(n-1),i+1}-{}\\
  \quad (2n-3)(4n-7)(2n-i)c_{2n,i-1},
  \end{array}\label{eq.r1}\\[1em]
&&\begin{array}{l}
  (n\!-\!2)(2n\!-\!5)(4n\!-\!11)(4n\!-\!7)(2n\!-\!i\!-\!2)(2n\!-\!i\!-\!1)c_{2n,i} = \\
  \quad (2n\!-\!5)(4n\!-\!11)(8i^2n^2-24i^2n+17i^2-16in^2+48in-{}\\
  \qquad 33i-16n^4+108n^3-258n^2+258n-92)c_{2(n-1),i}-{}\\
  \quad (n-1)(4n-7)(2n+i-5)(32in^2-122in+{}\\
  \qquad 117i-32n^3+168n^2-280n+144)c_{2(n-2),i}+{}\\
  \quad 6i(4i+1)(n-2)(n-1)(2n-3)(4n-7)c_{2(n-2),i+1}+{}\\
  \quad 36i(i+1)(n-2)(n-1)(2n-3)(4n-7)c_{2(n-2),i+2},
  \end{array}\label{eq.r2}\\[1em]
&&\begin{array}{l}
  18n(i-3)(i-2)(i-1)c_{2n,i} = \\
  \quad (2n+i-4)(10i^2n-24in^2-63in+i+16n^3+{}\\
  \qquad 76n^2+97n-3)c_{2n,i-3}-{}\\
  \quad 2(i-3)n(7i^2-12in-46i+33n+73)c_{2n,i-2}-{}\\
  \quad 3(i-3)(i-2)n(14i-12n-39)c_{2n,i-1}-{}\\
  \quad (2n-1)(4n-3)(2n-i+4)(2n-i+3)c_{2(n+1),i-3}.
  \end{array}\label{eq.r3}
\end{eqnarray}
}
When they are rewritten in operator notation, these recurrences form a
left Gr\"obner basis in the corresponding noncommutative operator
algebra, which is a bivariate polynomial ring (in the indeterminates
$S_i$ and $S_n$, denoting the forward shift operators
w.r.t.~$i$ and~$n$, respectively) with coefficients
in~$\Q(i,n)$. Together with the initial values
\[
  c_{2,1} = 1,\quad c_{2,2} = c_{2,3} = 0,\quad c_{4,1} = 2,
\]
they uniquely define the bivariate sequence~$(c_{2n,i})_{n,i\geq 1}$.
Note that the leading coefficients of~\eqref{eq.r1}, \eqref{eq.r2},
and \eqref{eq.r3} never vanish simultaneously in the region where
these recurrences are used to produce the values~$c_{2n,i}$ (in the
first quadrant, basically).

\subsection{Boundary Conditions}
\label{sec.bounds}
The sequence~$c$ is now extended to~$(c_{2n,i})_{n\geq 1,i\in\Z}$ and
it is proven that the assumption $c_{2n,i}=0$ for $i\leq0$ and for
$i\geq 2n$ is compatible with the recurrences
\eqref{eq.r1}--\eqref{eq.r3}. This knowledge will be useful for the
subsequent reasoning.

Provided with the appropriate initial conditions
($c_{2,0}=c_{4,0}=0$), it is obvious that the recurrence~\eqref{eq.r2}
produces zeros on the line $i=0$, since the terms $c_{2(n-2),i+1}$ and
$c_{2(n-2),i+2}$ vanish.  Similarly for $i=-1$, since the term
$c_{2(n-2),i+2}$ still vanishes; again assuming
$c_{2,-1}=c_{4,-1}=0$. Because of these two zero rows, it is clear
that everything beyond them (i.e., for $i<-1$) must be zero as well. A
simple computation shows that setting the initial conditions to~$0$ is
compatible with the recurrences~\eqref{eq.r1}--\eqref{eq.r3}.

Since the leading coefficient of~\eqref{eq.r1} does not vanish for any
integer point in the area $n\geq 2$ and $i\geq 2n$, this recurrence
can be used to produce the values of $c_{2n,i}$ in this area. The
support of~\eqref{eq.r1} indicates that only $c_{2n,2n}=0$ needs to be
shown. The first instances of this sequence are zero by construction,
and thus we have just to check that the third-order recurrence (not
printed here) that is automatically derived for $c_{2n,2n}$ does not
have a singularity in its leading coefficient; this is indeed not the
case. The univariate sequence $c_{2n,2n}$ is called the
\emph{diagonal} of the bivariate sequence~$c_{2n,i}$. Diagonals appear
frequently in combinatorial problems and their fast computation is a
topic of ongoing research in computer algebra. We used the command
\texttt{DFiniteSubstitute} of~\cite{Koutschan10b} here to perform the
substitution $i\to 2n$, which corresponds to the computation of the
diagonal.

It remains to show that $c_{2,i}=0$ for $i>2$, which is done in a similar
fashion.

\subsection{Identity {\large\eqref{eq.1}}}
Analogously to the computation of the diagonal in the previous
section, an annihilating operator for $c_{2n,2n-1}$ (of order~$4$, not
printed here) is obtained. Its leading coefficient has no nonnegative
integer roots, and it has the operator $S_n-1$ as a right
factor. Therefore it annihilates any constant sequence. The four
initial values are~$1$ by construction and therefore $c_{2n,2n-1}=1$
for all~$n\in\N$.

\subsection{Identity {\large\eqref{eq.2}}}
Once the implicit descriptions of the bivariate sequences $a_{i,j}$
and $c_{2n,i}$ are available, in terms of zero-dimensional left ideals
of recurrence operators, the summation identities~\eqref{eq.2}
and~\eqref{eq.3} are routinely provable, thanks to software packages
like \texttt{HolonomicFunctions}~\cite{Koutschan10b}.  The strategy is
as follows: first the closure properties of holonomic functions are
employed to compute recurrences for the product $c_{2n,i}a_{i,j}$; the
command \texttt{DFiniteTimes} does the job.  Then the method of
creative telescoping is invoked to produce some recurrences for the
left-hand side of~\eqref{eq.2} (this expression is denoted
by~$g_{n,j}$ in the following).  Two different algorithms for this
task are implemented in our package, namely the commands
\texttt{CreativeTelescoping} (Chyzak's algorithm~\cite{Chyzak00}) and
\texttt{FindCreativeTelescoping} (an alternative ansatz proposed by
the second author~\cite{Koutschan10c}). In order to prove
Identity~\eqref{eq.2} for instance, some operators of the form
\[
  P(j,n,S_j,S_n) + (S_i-1)Q(i,j,n,S_i,S_j,S_n)
\]
which annihilate the summand $c_{2n,i} a_{i,j}$ are computed.  It has
already been proven in Section~\ref{sec.bounds} that~$c_{2n,i}$ is
zero outside the summation range which implies that the sum runs over
natural boundaries. Therefore the \emph{principal parts} (or
\emph{telescopers}, denoted by $P$ above) of the creative telescoping
operators annihilate the sum, and the delta parts (denoted by~$Q$) can
be disregarded.  As a result we find
\[
\begin{split}
  & j(4n-7)(2n+j-2)g_{n,j} = {}\\
  & \qquad j(4n-3)(j-n+1)g_{n-1,j}+{}\\
  & \qquad (n-1)(4n-3)(2n-j-3)g_{n-1,j+1},\\[1em]
  & (j-2n)(2n+j-2)g_{n,j} = {}\\
  & \qquad 3(j-2)(j-1)g_{n,j-2} + (j-1)(2j-3)g_{n,j-1}.
\end{split}
\]
A close inspection reveals that only the initial values $g_{1,1}$,
$g_{2,1}$, and $g_{2,2}$ need to be given, if the above recurrences
shall be used to compute all values of $g_{n,j}$ for $n\geq 1$ and
$1\leq j<2n$. A simple calculation shows that they are all zero,
concluding the proof of~\eqref{eq.2}.

Note that the above reasoning is somehow about the maximal possible
area: if one tries to extend it further, i.e., to show that
$g_{n,j}=0$ in the whole first quadrant, the first step being the
points $j=2n$, then the second recurrence, the only one that is
applicable in this case, breaks down. Indeed, the values~$g_{n,2n}$
are nonzero as is demonstrated in the next section.

\subsection{Identity {\large\eqref{eq.3}}}\label{sec.eq3}
Identity~\eqref{eq.3} is done in a very similar fashion, using the
method of creative telescoping. Again the summation is over natural
boundaries. Thus the principal part of the computed creative
telescoping operator gives rise to a recurrence for the left-hand side
of~\eqref{eq.3} which is denoted by~$r_n$ here:
\[
  \begin{array}{l}
   2(4n-11)(4n-7)(4n-5)(7n-13)r_n = \\
   \qquad (4n-11)(350n^3-1413n^2+1798n-714)r_{n-1}-{}\\
   \qquad 9(n-2)(2n-3)(4n-7)(7n-6)r_{n-2}.
  \end{array}
\]
For $n=1$ and $n=2$ the summation in~\eqref{eq.3} yields the initial
values $r_1=1$ and $r_2=5$. It is easily verified that the unique
solution of the above recurrence is $r_n=4n-3$. Since
$r_n=b_{2n}/b_{2n-2}$ gives the quotients of two consecutive Pfaffians
$b_{2n}=\Pf(a_{i,j})_{1\leq i,j\leq 2n}$, it follows that
\[
  b_{2n} = \prod_{k=1}^n\frac{b_{2k}}{b_{2k-2}} = \prod_{k=1}^n(4k-3) = \prod_{k=0}^{n-1}(4k+1).
\]
This concludes the proof of Theorem~\ref{thm.pfMotz}.

%
%
\section{A Delannoy Number Pfaffian}\label{sec.Dela}
We now consider a Pfaffian that appears as Formula (6.4)
in~\cite{IshikawaTagawaZeng10}, again as an open problem.
\begin{theorem}
Let 
\[
  D_n=\sum_{k=0}^n\binom{n}{k}\binom{n+k}{k}
\]
denote the $n$-th central Delannoy number. Then for all integers
$n\geq 1$ the following identity holds:
\begin{equation}\label{eq.pfDela}
  \Pf\big((j-i)D_{i+j-3}\big)_{1\leq i,j\leq 2n} = 2^{(n+1)(n-1)}(2n-1)\prod_{k=1}^{n-1}(4k-1).
\end{equation}
\end{theorem}
\begin{proof}
The auxiliary function~$c_{2n,i}$ in this example can be defined by
the following recurrences (plus a sufficient amount of initial
values): 
{\setlength{\arraycolsep}{0pt}
\begin{eqnarray*}
&&\begin{array}{l}
   2(i-3)(i-2)(i-1)c_{2n,i} = \\
   \quad 3(i-3)(i-2)(8i-27)c_{2n,i-1}-{}\\
   \quad (i-3)(76i^2-589i-8n^2+16n+1109)c_{2n,i-2}+{}\\
   \quad 3(8i^3-105i^2-16in^2+32in+443i+68n^2-{}\\
   \qquad 136n-600)c_{2n,i-3}-{}\\
   \quad (2i-11)(i-2n-3)(i+2n-7)c_{2n,i-4},
  \end{array}\\[1em]
&&\begin{array}{l}
   2(n-2)(2n-3)(4n-9)(i-2n+1)(i-2n+2)c_{2n,i} = \\
   \quad (n-1)(i+2n-5)(68i^2n-102i^2-96in^2+178in-{}\\
   \qquad 43i+64n^3-208n^2+200n-56)c_{2(n-1),i}-{}\\
   \quad 6i(n\!-\!1)(2n-3)(35i^2+4in-66i-n+14)c_{2(n-1),i+1}+{}\\
   \quad i(i+1)(n-1)(2n-3)(70i+4n-31)c_{2(n-1),i+2}-{}\\
   \quad 6i(i+1)(i+2)(n-1)(2n-3)c_{2(n-1),i+3}.
  \end{array}
\end{eqnarray*}
} The proof is very analogous to the one of Theorem~\ref{thm.pfMotz},
see the accompanying Mathematica notebook for the details.
\end{proof}

%
%
\section{A Narayana Number Pfaffian}\label{sec.Nara}
The following Pfaffian appears as Formula~(6.6) in~\cite{IshikawaTagawaZeng10}:
\begin{theorem}\label{thm.pfNara}
Let $N_n(x)$ denote the $n$-th Narayana polynomial defined by
\begin{eqnarray*}
  N_0(x) & = & 1,\\
  N_n(x) & = & \sum_{k=0}^n\frac{1}{n}\binom{n}{k}\binom{n}{k-1}x^k,\quad (n\geq 1).
\end{eqnarray*}
Then for all $n\geq0$ the following identity holds:
\begin{equation}\label{eq.pfNara}
  \Pf\big((j-i)N_{i+j-2}(x)\big)_{1\leq i,j\leq 2n} = x^{n^2}\prod_{k=0}^{n-1}(4k+1).
\end{equation}
\end{theorem}
\begin{proof}
Again, the proof of this evaluation is analogous to the previous ones
of~\eqref{eq.pfMotz} and~\eqref{eq.pfDela}, see the accompanying
Mathematica notebook for the details. The main difference is that now
the free parameter~$x$ is involved, which on the one hand makes the
computations and the intermediate results more voluminous. On the
other hand, some arguments in the proof (like ``the leading
coefficient of some recurrence is never zero'') become more intricate.

One solution to address the latter issue is to argue that $x$ is a
formal parameter; then any polynomial in $x$ which is not identically
zero, is considered to be nonzero (as an element in the corresponding
polynomial ring). If one feels uneasy about this argument, one can as
well try to find conditions under which all steps of the proof are
sound; for our reasoning the assumption $x<-1$ was sufficient. Hence
the evaluation is proven only for $x<-1$.  But for specific~$n$, the
Pfaffian is a polynomial in~$x$ (of a certain degree), as well as the
evaluation on the right-hand side of~\eqref{eq.pfNara}. Thus their
difference is a polynomial in~$x$ which has been proven to be zero for
all $x<-1$. By the fundamental theorem of algebra it follows that this
polynomial is identically zero, and therefore the evaluation of the
Pfaffian is true for all complex numbers~$x$.
\end{proof}

\begin{corollary}\label{thm.pfSchr}
Let 
\[
  S_n=\sum_{k=0}^n \frac{1}{k+1}\binom{n+k}{2k}\binom{2k}{k}
\]
denote the (large) Schr\"oder numbers. Then for all integers $n\geq0$
the following identity holds:
\[
  \Pf\big((j-i)S_{i+j-2}\big)_{1\leq i,j\leq 2n} = 2^{n^2}\prod_{k=0}^{n-1}(4k+1).
\]
\end{corollary}
\begin{proof}
This identity follows from Theorem~\ref{thm.pfNara} and the equality
$S_n=N_n(2)$; the latter fact can be easily proven from the
definitions of these quantities using Zeilberger's algorithm, for
example.
\end{proof}

In~\cite{IshikawaTagawaZeng10} it has already been noted that
Theorem~\ref{thm.pfNara} implies the Pfaffian of
Corollary~\ref{thm.pfSchr} involving the Schr\"oder numbers.
Similarly, it is stated there that also the Pfaffian~\eqref{eq.pfMotz}
is a special case of Theorem~\ref{thm.pfNara}. However, in order to
reflect the historic evolution of our results and for reasons of a
clear presentation, we included Theorem~\ref{thm.pfMotz} and its
detailed proof in this article.

%
%
\section{Application of Theorem 2} 
\label{sec.Appl}

Let $A=(a_{i,j})_{1\leq i\leq n,\,j\geq1}$ be any $n$-rowed matrix.
If $J=\left\{j_{1},\dots,j_{n}\right\}$ 
is a set of column indices,
then we write $A_{J}=A_{j_{1},\dots,j_{n}}$
for the square submatrix of size~$n$ obtained from $A$ by choosing 
the columns indexed by~$J$.
If $A=(a_{i,j})_{i,j\geq1}$ is a matrix with
infinitely many rows and columns,
and $I=\left\{i_{1},\dots,i_{n}\right\}$  (resp. $J=\left\{j_{1},\dots,j_{n}\right\}$)
is a set of row (resp. column) indices,
then let $A^{I}_{J}=A^{i_{1},\dots,i_{n}}_{j_{1},\dots,j_{n}}$
denote the square submatrix of size $n$ obtained from $A$ by choosing 
the rows~$I$ and the columns~$J$.

A \defterm{partition} is a nonincreasing sequence 
$\lambda=(\lambda_{1},\lambda_{2},\dots)$
of nonnegative integers with only finitely many nonzero elements.
The number of nonzero elements 
in $\lambda$ is called the \defterm{length} of $\lambda$ and
 is denoted by $l(\lambda)$.
An \defterm{odd partition} is a partition with odd parts 
and an \defterm{even partition} is a partition with even parts.
The \defterm{conjugate} of $\lambda$ is the partition
$\lambda'=(\lambda_1',\lambda_2',\dots)$,
where $\lambda_i'$ is the number defined by
$
\lambda_{i}'=\#\{j\,|\,\lambda_{j}\geq i\}.
$
Given a partition $\lambda$ such that $l(\lambda)\leq n$,
let $I_n(\lambda)$ denote the $n$-element set of nonnegative integers
defined by
\[
I_{n}(\lambda)=\left\{\lambda_{n}+1,\lambda_{n-1}+2,\dots,\lambda_{1}+n\right\}.
\]
For example, $\lambda=(3,3,1,1)$ is an odd partition of length~$4$,
and $I_{4}(\lambda)=\{2,3,6,7\}$.
The conjugate of $\lambda$ equals $(4,2,2)$,
which is an even partition.

Let 
 $H(n)=\left(h(i,j)\right)_{1\leq i\leq n,\,j\geq1}$
denote the $n$-rowed matrix
whose entries are given by
\begin{align}
& h(i,2k-1) = \left(\!\genfrac{}{}{0pt}{}{i-1}{k-1}\!\right)\,{}_{2}F_{1}\!
\left(\genfrac{}{}{0pt}{}{\frac{k-i}2,\frac{k-i+1}2}{k+1};4\right), \label{eq:h-2k-1}\\
& h(i,2k) = (i-1)\left(\!\genfrac{}{}{0pt}{}{i-2}{k-1}\!\right)\,{}_{2}F_{1}\!
\left(\genfrac{}{}{0pt}{}{\frac{k-i+1}2,\frac{k-i+2}2}{k+1};4\right).\label{eq:h-2k}
\end{align}
In fact $h(i,2k-1)$ is the number of Motzkin paths from $(0,0)$ to $(i-1,k-1)$.
We also note that $h(i,2k)=k[x^{i+k-1}](1+x+x^2)^{i-1}$,
where $[x^n]f(x)$ denotes the coefficient of $x^n$ in a polynomial $f(x)$.
For example,
if $n=4$, then we have
\[
H(4)=\begin{pmatrix} 
1&0&0&0&0&0&0&0&\hdots\\
1&1&1&0&0&0&0&0&\hdots\\
2&2&2&2&1&0&0&0&\hdots\\
4&6&5&6&3&3&1&0&\hdots
\end{pmatrix}.
\]
For example,
$h(4,3)=5$ gives the number of Motzkin paths from $(0,0)$ to $(3,1)$:
\[
UUD,\quad UHH,\quad UDU,\quad HUH,\quad HHU.
\]
Meanwhile, $h(4,4)=6$ equals $2$ times the coefficient of $x^5$ in $(1+x+x^2)^3$.
The main purpose of this section is to give a proof of the following theorem
as a corollary of Theorem~\ref{thm.pfMotz}.
\begin{theorem}
\label{th:main-application}
Let $n$ be a positive integer, and let $H(n)$ be as above.
Then we have
\begin{align}
\sum_{\genfrac{}{}{0pt}{}{\lambda}{\lambda,\lambda'\text{ even}}}
\det H(2n)_{I_{2n}(\lambda)}=\prod_{k=0}^{n-1}(4k+1), 
\label{eq:application}
\end{align}
where the sum on the left-hand side runs over all even partitions $\lambda$ such that
$\ell(\lambda)\leq2n$ and $\lambda'$ is also even.
\end{theorem}
We notice that this theorem is a consequence of an addition formula for ${}_2F_1$ and Theorem~\ref{thm.pfMotz}.
But it is not so easy to find a lattice path interpretation of $\det H(2n)_{I_{2n}(\lambda)}$
since we do not know a lattice path interpretation of $h(i,2k)$. 
To prove the theorem, we cite the following two lemmas 
from~\cite{IshikawaTagawaZeng09} and~\cite{IshikawaWakayama10}.
\begin{lemma}
If $i$ and $j$ are nonnegative integers, then we have
\begin{multline}
\label{okinawa}
  \sum_{k\geq 0}\left(\!\genfrac{}{}{0pt}{}{i}{k}\!\right)\!\!\left(\!\genfrac{}{}{0pt}{}{j}{k}\!\right)
  {}_2F_1\!\left(\genfrac{}{}{0pt}{}{\frac{k-i+1}{2},\frac{k-i}{2}}{k+2};4\right)
  {}_2F_1\!\left(\genfrac{}{}{0pt}{}{\frac{k-j+1}{2},\frac{k-j}{2}}{k+2};4\right)=\\
  {}_2F_1\!\left(\genfrac{}{}{0pt}{}{\frac{1-i-j}{2},\frac{-i-j}{2}}{2};4\right).
\end{multline}
\end{lemma}
\begin{proof}
The proof can be found in~\cite[Lemma 5.2]{IshikawaTagawaZeng09}.
\end{proof}
%
%
%
%
%
%
\begin{lemma}
\label{th:msf}
For $n\in\N$ let $T=(t_{i,j})_{1\le i\le 2n,\, j\geq1}$ be an $n$-rowed matrix,
and let $A=(a_{i,j})_{i,j\geq1}$ be a skew-symmetric matrix with infinitely many rows and columns,
i.e. $a_{j,i}=-a_{i,j}$ for $i,j\geq1$.
Then we have
\begin{align}
\label{eq_msf}
\sum_{\genfrac{}{}{0pt}{}{I}{\sharp I=2n}}
\Pf(A^{I}_I) \det(T_I)&
=\Pf(Q),
\end{align}
where the sum on the left-hand side runs over all $2n$-element sets of positive integers
and the skew-symmetric matrix~$Q$ is defined by $Q=(Q_{i,j})=TA\,T^\mathrm{T}$ 
whose entries may be written in the form
\begin{equation}
\label{pf_msf}
Q_{i,j}=\sum_{1\le k<l} a_{k,l} \det(T^{i,j}_{k,l}),
\qquad(1\le i,j\le n).
\end{equation}
\end{lemma}
\begin{proof}
The proof of this minor summation formula can be found in~\cite[Theorem~3.2]{IshikawaWakayama10}.
\end{proof}
\begin{proof}[of Theorem~\ref{th:main-application}]
Set 
\[
  a_{i,j}=\begin{cases}
   1  & \text{if $i=2k-1$ and $j=2k$ for some $k\in\N$,}\\
   -1 & \text{if $i=2k$ and $j=2k-1$ for some $k\in\N$,}\\
   0  & \text{otherwise.}
  \end{cases}
\]
and $t_{i,j}=h(i,j)$ in~\eqref{eq_msf}
where $h(i,j)$ is as defined in~\eqref{eq:h-2k-1} and~\eqref{eq:h-2k}.
Then one can show by direct calculation $Q_{i,j}$ in~\eqref{pf_msf} is given by
\begin{align*}
Q_{i,j}&=\sum_{k\geq1}
\det\begin{pmatrix}h(i,2k-1)&h(i,2k)\\h(j,2k-1)&h(j,2k)\end{pmatrix}\\
& = (j-1)\sum_{k\geq1}\left(\!\genfrac{}{}{0pt}{}{i-1}{k-1}\!\right)\left(\!\genfrac{}{}{0pt}{}{j-2}{k-1}\!\right)
{}_2F_1\!\left(\genfrac{}{}{0pt}{}{\frac{k-i}{2},\frac{k-i+1}{2}}{k+1};4\right)\\
& \qquad\qquad\qquad\times {}_2F_1\!\left(\genfrac{}{}{0pt}{}{\frac{k-j+1}{2},\frac{k-j+2}{2}}{k+1};4\right)-{}\\
& \quad\,\, (i-1)\sum_{k\geq1}\left(\!\genfrac{}{}{0pt}{}{i-2}{k-1}\!\right)\left(\!\genfrac{}{}{0pt}{}{j-1}{k-1}\!\right)
{}_2F_1\!\left(\genfrac{}{}{0pt}{}{\frac{k-i+1}{2},\frac{k-i+2}{2}}{k+1};4\right)\\
& \qquad\qquad\qquad\times{}_2F_1\!\left(\genfrac{}{}{0pt}{}{\frac{k-j}{2},\frac{k-j+1}{2}}{k+1};4\right)
\end{align*}
By \eqref{okinawa} we obtain
\begin{align*}
Q_{i,j}=(j-i)
{}_2F_1\!\left(\genfrac{}{}{0pt}{}{\frac{3-i-j}{2},\frac{4-i-j}{2}}{2};4\right)
=(j-i)M_{i+j-3}
\end{align*}
which, using~\eqref{eq_msf}, gives
\[
\sum_{\genfrac{}{}{0pt}{}{I}{\#I=2n}}\Pf\!\big(A^I_I\big)\, \det H(2n)_{I} = \Pf\!\big((j-i)M_{i+j-3}\big)_{1\leq i,j\leq2n}.
\]
It is not hard to see that $\Pf(A^I_I)=1$ if $I=I_{2n}(\lambda)$ for a partition $\lambda$
such that $\ell(\lambda)\leq2n$ and $\lambda$, $\lambda'$ are even,
and $\Pf(A^I_I)=0$ otherwise (see \cite{IshikawaWakayama10}).
Hence we obtain the desired formula \eqref{eq:application} as a consequence of Theorem~\ref{thm.pfMotz}.
\end{proof}
We can regard the numbers 
\[
  h(i,2k-1)=\left(\!\genfrac{}{}{0pt}{}{i-1}{k-1}\!\right)\,
  {}_{2}F_{1}\!\left(\genfrac{}{}{0pt}{}{\frac{k-i}2,\frac{k-i+1}2}{k+1};4\right)
\]
as a generalization of the Motzkin numbers $M_n$
since they count the Motzkin paths from $(0,0)$ to $(i-1,k-1)$,
and write $\mathcal{M}^{(k)}_{i}=h(i,2k-1)$ hereafter.
In fact $\big(\mathcal{M}^{(k)}_{i}\big)_{i\geq1}$ gives 
the $(k-1)$-th column of the Motzkin triangle~\cite{DonagheyShapiro77}.
Note that $M_n=\mathcal{M}^{(1)}_{n+1}$ so that
Theorem~\ref{thm.pfMotz} reads 
\[
  \Pf\!\big((j-i)\mathcal{M}^{(1)}_{i+j-2}\big) = \prod_{k=0}^{n-1}(4k+1).
\]
It may now be attractive to present a generalization of Theorem~\ref{thm.pfMotz} as follows.
\begin{conjecture}\label{conj.gen}
Let $n$ and $k$ be positive integers. 
\begin{enumerate}
\item[(i)]
Then the Pfaffian
\[
  \Pf\!\left((j-i)\mathcal{M}^{(k)}_{i+j-2}\right)_{1\leq i,j\leq2n}
\]
equals
\[
  \prod_{i=0}^{m-1}\prod_{j=0}^{k-1}(4ki+2j+k)
\]
if $m=n/k$ is an integer, and it equals 
\[
  \Bigg(\prod_{j=1}^{\lfloor k/2\rfloor}\frac{1}{2j-k}\Bigg)\Bigg(\prod_{i=0}^{m-1}\prod_{j=1}^{k}(4ki+2j-k)\Bigg)
\]
if $k$ is odd and $m=\big(n+\lfloor k/2\rfloor\big)\big/k$ is an integer. 
The Pfaffian is zero in all other cases.
\item[(ii)]
Meanwhile, the Pfaffian
\[
  \Pf\!\left((j-i)\left(\mathcal{M}^{(k)}_{i+j-2}+\mathcal{M}^{(k)}_{i+j-1}\right)\right)_{1\leq i,j\leq2n}
\]
equals
\[
  \prod_{i=0}^{m-1}\prod_{j=0}^{k-1}(4ki+2j+k+1)
\]
if $m=n/k$ is an integer, and it equals 
\[
  \Bigg(\prod_{j=1}^{k/2}\frac{1}{2j-k-1}\Bigg)\Bigg(\prod_{i=0}^{m-1}\prod_{j=1}^{k}(4ki+2j-k-1)\Bigg)
\]
if $k$ is even and $m=\big(n+k/2\big)\big/k$ is an integer. 
The Pfaffian is zero in all other cases.
\end{enumerate}
\end{conjecture}
We want to remark that for $k=1$ part (i) is just
Theorem~\ref{thm.pfMotz} and part (ii) can be proven
analogously. Unfortunately these two Pfaffians are periodically zero
if $k\geq2$. This prevents us from applying our method to the
conjecture, since we consider the quotient of two consecutive
Pfaffians. Of course, one could come up with a Pfaffian analogue of
the double-step method presented in~\cite{KoutschanThanatipanonda12},
which would settle Conjecture~\ref{conj.gen} for the special case
$k=2$.  This construction may be extended for $k=3$, $k=4$, etc., at
the cost of more and more involved computations. However, this
approach will not work for symbolic~$k$ in general.

Conjecture~\ref{conj.gen} can be regarded as a Pfaffian analogue of
the Hankel determinants of Motzkin numers \cite[Proposition~2]{Aigner98},
and Hankel determinants of sums of two consecutive Motzkin numbers \cite[Theorem~3.2]{CameronYip11}.
Many combinatorial arguments are known for the Hankel determinants,
but little is known for Hankel Pfaffians (see \cite{IshikawaTagawaZeng10,Lascoux11}).
It may be interesting to discover a combinatorial reason 
why we can expect such a nice formula for the Hankel Pfaffians of (sums of) Motzkin numbers.

\section{Acknowledgments} The authors are grateful to Jiang Zeng
for initiating their first contact during the international conference
on asymptotics and special functions in Hongkong, and to the anonymous
referees for their diligent work.  The second named author was
employed by the Research Institute for Symbolic Computation (RISC) of
the Johannes Kepler University in Linz, Austria, while carrying out
the research for the present paper.


\begin{thebibliography}{10}

\bibitem{Aigner98}
Martin Aigner.
\newblock {M}otzkin numbers.
\newblock {\em European Journal Combinatorics}, 19:663--675, 1998.

\bibitem{CameronYip11}
Naiomi~T. Cameron and Andrew~C.M. Yip.
\newblock {H}ankel determinants of sums of consecutive {M}otzkin numbers.
\newblock {\em Linear Algebra and its Applications}, 434:712--722, 2011.

\bibitem{Chyzak00}
Fr\'{e}d\'{e}ric Chyzak.
\newblock An extension of {Z}eilberger's fast algorithm to general holonomic
  functions.
\newblock {\em Discrete Mathematics}, 217(1-3):115--134, 2000.

\bibitem{DonagheyShapiro77}
Robert Donaghey and Louis~W. Shapiro.
\newblock Motzkin numbers.
\newblock {\em Journal of Combinatorial Theory, Series A}, 23:291--301, 1977.

\bibitem{IshikawaTagawaZeng09}
Masao Ishikawa, Hiroyuki Tagawa, and Jiang Zeng.
\newblock A $q$-analogue of {C}atalan {H}ankel determinants.
\newblock {\em RIMS K\^oky\^uroku Bessatsu}, B11:19--42, 2009.
\newblock arXiv:1009.2004.

\bibitem{IshikawaTagawaZeng10}
Masao Ishikawa, Hiroyuki Tagawa, and Jiang Zeng.
\newblock {P}faffian decomposition and a {P}faffian analogue of $q$-{C}atalan
  {H}ankel determinants.
\newblock Technical Report 1011.5941, arXiv, 2010.

\bibitem{IshikawaWakayama10}
Masao Ishikawa and Masato Wakayama.
\newblock Applications of the minor summation formula {III}: {P}l\"{u}cker
  relations, lattice paths and {P}faffians.
\newblock {\em Journal of Combinatorial Theory, Series A}, 113:113--155, 2006.

\bibitem{Kauers09}
Manuel Kauers.
\newblock Guessing handbook.
\newblock Technical Report 09-07, RISC Report Series, Johannes Kepler
  University Linz, 2009.
\newblock http:/$\!$/www.risc.jku.at/\linebreak[0]%
  research/\linebreak[0]combinat/\linebreak[0]%
  software/\linebreak[0]Guess/

\bibitem{Koutschan09}
Christoph Koutschan.
\newblock {\em Advanced Applications of the Holonomic Systems Approach}.
\newblock PhD thesis, RISC, Johannes Kepler University, Linz, Austria, 2009.

\bibitem{Koutschan10c}
Christoph Koutschan.
\newblock A fast approach to creative telescoping.
\newblock {\em Mathematics in Computer Science}, 4(2-3):259--266, 2010.

\bibitem{Koutschan10b}
Christoph Koutschan.
\newblock {HolonomicFunctions (User's Guide)}.
\newblock Technical Report 10-01, RISC Report Series, Johannes Kepler
  University Linz, 2010.
\newblock http:/$\!$/www.risc.jku.at/\linebreak[0]%
  research/\linebreak[0]combinat/\linebreak[0]%
  software/\linebreak[0]HolonomicFunctions/

\bibitem{KoutschanKauersZeilberger10}
Christoph Koutschan, Manuel Kauers, and Doron Zeilberger.
\newblock Proof of {G}eorge {A}ndrews's and {D}avid {R}obbins's $q$-{TSPP}
  conjecture.
\newblock {\em Proceedings of the US National Academy of Sciences},
  108(6):2196--2199, 2011.

\bibitem{KoutschanThanatipanonda12}
Christoph Koutschan and Thotsaporn Thanatipanonda.
\newblock Advanced computer algebra for determinants.
\newblock {\em Annals of Combinatorics}, 2012. To appear, preprint in arXiv:1112.0647.

\bibitem{Krattenthaler99}
Christian Krattenthaler.
\newblock Advanced determinant calculus.
\newblock {\em S\'{e}minaire Lothar\-ingien de Combinatoire}, 42:1--67, 1999.
\newblock Article B42q.

\bibitem{Lascoux11}
Alain Lascoux.
\newblock {H}ankel {P}faffians, discriminants and {K}azhdan-{L}usztig bases.
\newblock Technical Report 1103.4971, arXiv, 2011.

\bibitem{Zeilberger90}
Doron Zeilberger.
\newblock A holonomic systems approach to special functions identities.
\newblock {\em Journal of Computational and Applied Mathematics},
  32(3):321--368, 1990.

\bibitem{Zeilberger07}
Doron Zeilberger.
\newblock {The HOLONOMIC ANSATZ II. Automatic DISCOVERY(!) and PROOF(!!) of
  Holonomic Determinant Evaluations}.
\newblock {\em Annals of Combinatorics}, 11:241--247, 2007.

\end{thebibliography}

\end{document}